\documentclass[a4paper,12pt,intlimits,oneside]{amsart}
\usepackage{amsmath}
\usepackage{amsthm}
\usepackage{latexsym}
\usepackage{amssymb}
\usepackage{xcolor}
\numberwithin{figure}{section}
\def\R{{\mathbb R}}
\def\C{{\mathbb C}}
\def\T{{\mathbb T}}
\def\Z{{\mathbb Z}}

\def\s{\vskip 0.25cm\noindent}
\def\e{\varepsilon}
\def\build#1_#2^#3{\mathrel{
\mathop{\kern 0pt#1}\limits_{#2}^{#3}}}
\def\td_#1,#2{\mathrel{\mathop{\build\longrightarrow_{#1\rightarrow #2}^{}}}}
\newtheorem{theorem}{Theorem}[section]
\newtheorem{corollary}{Corollary}
\newtheorem{proposition}{Proposition}
\newtheorem{lemma}{Lemma}

\begin{document}
%\color{blue}
\title[Effective dynamics for some non linear wave equation]{Effective integrable dynamics for some nonlinear wave equation}
\author{Patrick G\'erard}
\address{Universit\'e Paris-Sud XI, Laboratoire de Math\'ematiques
d'Orsay, CNRS, UMR 8628} \email{{\tt Patrick.Gerard@math.u-psud.fr}}

\author[S. Grellier]{Sandrine Grellier}
\address{MAPMO-UMR 6628,
D\'epartement de Math\'ematiques, Universit\'e d'Orleans, 45067
Orl\'eans Cedex 2, France} \email{{\tt
Sandrine.Grellier@univ-orleans.fr}}


\date{Oct 24, 2011}

\begin{abstract}
We consider the following degenerate half wave equation on the one dimensional  torus
$$\quad i\partial _t u-|D|u=|u|^2u, \; u(0,\cdot)=u_0.
$$
We show that, on a large time interval, the solution may be approximated by the solution of a completely integrable system-- the cubic Szeg\"o equation. As a consequence, we prove an instability result for large $H^s$ norms of solutions of this wave equation.\end{abstract}

\maketitle

\section{Introduction}

Let us consider, on the one dimensional torus $\T$, the following``half-wave" equation
\begin{equation}\label{(W)}\quad i\partial _t u-|D|u=|u|^2u, \; u(0,\cdot)=u_0.
\end{equation}
Here $|D|$ denotes the pseudo-differential operator defined by
$$|D|u=\sum |k|u_ke^{ikx}, \quad u=\sum_k u_ke^{ikx}.$$
This equation can be seen as a toy model  for non linear Schr\"odinger equation on degenerate geometries leading to lack of dispersion. For instance, it has the same structure as the cubic non linear Schr\"odinger equation on the Heisenberg group, or associated with the Gru\v{s}in operator. We refer to \cite{GGX} and \cite{GG} for more detail.

We endow $L^2(\T )$ with the symplectic form
$$\omega (u,v)={\rm Im}(u|v)\ .$$
Equation (\ref{(W)}) may be seen as the Hamiltonian system  related to the energy function
$H(u):= \frac 12(|D|u,u)+\frac 14\Vert u\Vert_{L^4}^4$. In particular, $H$ is invariant by the flow which also admits the following conservation laws, 
$$Q(u):=\Vert u\Vert_{L^2}^2,\quad M(u):=(Du\vert u).$$
However, equation (\ref{(W)})  is a non dispersive equation. Indeed, it is equivalent to the system
\begin{equation}\label{system}
 i(\partial _t \pm \partial_x) u_\pm=\Pi_\pm(|u|^2u),\; u_\pm(0,\cdot)=\Pi_\pm(u_0),
 \end{equation}
where $u_\pm=\Pi_\pm(u)$. 
Here, $\Pi_+$ denotes the orthogonal projector from $L^2(\T )$ onto $$L^2_+(\T ):=\{u=\sum_{k\ge 0} u_k e^{ikx},\; (u_{k})_{k\ge 0}\in\ell^2 \}$$ and $\Pi _-:=I-\Pi _+$. 

Though the scaling is $L^2$-critical, the first iteration map of the Duhamel formula
$$u(t)=e^{-it|D|}u_0-i\int_0^t e^{-i(t-\tau)|D|}(|u(\tau)|^2u(\tau))d\tau$$ is not bounded on $H^s$ for $s<\frac 12$.
Indeed, such boundedness would require the inequality
$$\int_0^1\Vert {\rm e}^{-it\vert D\vert }f\Vert _{L^4(\T )}^4\, dt \lesssim \Vert f\Vert _{H^{s/2}}^4.$$
However, testing this inequality on functions localized on positive modes for instance, shows that this fails if $s<\frac 12$  (see the appendix for more detail).

Proceeding as in the case of  the cubic Szeg\"o equation (see \cite{GG}, Theorem 2.1), \begin{equation}\label{Szego}i\partial _tw=\Pi_+(|w|^2w),\end{equation}
one can prove the global existence and uniqueness of solutions of (\ref{(W)}) in $H^s$ for any $s\ge 1/2$. The proof uses in particular the a priori bound  of the $H^{1/2}$-norm provided by the energy conservation law. \s

\begin{proposition}\label{Propo}
Given $u_0\in H^{\frac 12}(\T )$, there exists $u\in C(\R ,H^{\frac 12}(\T ))$ unique such that
$$i\partial _tu-\vert D\vert u=\vert u\vert ^2 u\ ,\ u(0,x)=u_0(x)\ .$$
Moreover if $u_0\in H^s(\T )$ for some $s>\frac 12$, then 
$u\in C(\R , H^s(\T ))$. 
\end{proposition}
 Notice that similarly to the cubic Szeg\"o equation, the proof  of Proposition \ref{Propo} provides only bad large time estimates,
$$\Vert u(t)\Vert _{H^s}\lesssim {\rm e}^{{\rm e}^{C_st}}.$$
This naturally leads to the question of the large time behaviour of solutions of (\ref{(W)}).
In order to answer this question, a fundamental issue is  the decoupling of non negative and negative modes in system (\ref{system}). 
Assuming that initial data are small and spectrally localized on non negative modes, a first step in that direction is given by the next simple proposition,  which shows that $u_-(t)$ remains smaller in $H^{1/2}$ uniformly in time.
\begin{proposition}
Assume
$$\Pi _+u_0=u_0=O(\varepsilon )\text{ 
in }H^{\frac 12}(\T).$$ Then, the solution $u$ of (\ref{(W)}) satisfies
$$\sup _{t\in \R }\Vert \Pi _-u(t)\Vert _{H^{\frac 12}}=O(\varepsilon ^2)\ .$$
\end{proposition}

\begin{proof}

By the energy and momentum conservation laws, we have
\begin{eqnarray*}
(\vert D\vert u, u)+\frac 12 \Vert u\Vert _{L^4}^4&=& (\vert D\vert u_0, u_0)+\frac 12 \Vert u_0\Vert _{L^4}^4\ ,\\
(Du,u)&=&(Du_0,u_0)\ .
\end{eqnarray*}

Substracting these equalities, we get 
$$2(\vert D\vert u_-, u_-)+\frac 12\Vert u\Vert _{L^4}^4=\frac 12 \Vert u_0\Vert _{L^4}^4=O(\varepsilon ^4)\  ,$$
hence
$$\Vert u_-\Vert _{H^{\frac 12}}^2=O(\varepsilon ^4)\ .$$
\end{proof}

This decoupling result suggests to neglect $u_-$ in  system (\ref{system}) and hence to compare the solutions of (\ref{(W)}) to the solutions of 
$$i\partial _t v-Dv=\Pi_+(|v|^2v),
$$
which can be reduced to (\ref{Szego}) by the transformation $v(t,x)=w(t,x-t)$. 

Our main result is the following.
\begin{theorem}\label{Approx}
Let $s>1$ and $u_0=\Pi_+(u_0)\in L^2_+(\T)\cap H^s(\T)$ with $\Vert u_0\Vert_{H^s}=\varepsilon$, $\varepsilon>0$ small enough.
Denote by  $v$  the solution of the cubic Szeg\"o equation
\begin{equation}\label{CauchySzego} i\partial_tv-Dv=\Pi_+(|v|^2v) \ ,\ 
v(0,\cdot)=u_0 .
\end{equation}
Then, for any $\alpha >0$,  there exists  a constant $c=c_\alpha <1$ so that
\begin{equation}\label{approximation}\Vert u(t)-v(t)\Vert_{H^s}=\mathcal O(\varepsilon^{3-\alpha} )\text{ for }t\le \frac {c_{\alpha }}{\varepsilon^2}\log\frac 1\varepsilon\ .\end{equation}
Furthermore, there exists $c>0$ such that 
\begin{equation}\label{estimateLinfty}\forall t\le \frac c{\e ^3}\ ,\ \Vert u(t)\Vert _{L^\infty} =\mathcal O(\e ) \ .
\end{equation}
\end{theorem}

Theorem \ref{Approx} calls for several remarks. Firstly,
if we rescale $u$ as $\e u$, equation (\ref{(W)})  becomes
$$\quad i\partial _t u-|D|u=\e ^2|u|^2u, \; u(0,\cdot)=u_0$$
with $\Vert u_0\Vert _{H^s}=1 $. On the latter equation, it is easy to prove that $u(t)=e^{-it|D|}u_0+o(1)$ for $t<\!\!<\frac 1{\e^2}$, so that non linear effects only start for $\frac 1{\e^2}\lesssim t$. Rescaling $v$ as $\e v$ in equation (\ref{CauchySzego}), Theorem \ref{Approx} states that the cubic Szeg\"o dynamics appear as  the effective dynamics of equation (\ref{(W)}) on a time interval where  non linear effects are taken into account.
\s
 Secondly, as pointed out before, (\ref{CauchySzego}) reduces to (\ref{Szego}) by a simple Galilean transformation. Equation (\ref{Szego}) has been studied in \cite{GGX}, \cite{GG} and \cite{GG2} where its complete integrability is established together with an explicit formula for its generic solutions. Consequently, the first part of Theorem \ref{Approx} provides an accurate description of solutions of equation (\ref{(W)}) for a reasonably large time.
Moreover, the second part of Theorem \ref{Approx} claims an $L^\infty$ bound for the solution of (\ref{(W)}) on an even larger time. This latter bound is closely related to a special conservation law of equation (\ref{Szego}), namely, some Besov norm of $v$ --see section \ref{LaxPair} below. 
\s
Our next observation is that, in the case of small Cauchy data localized on non negatives modes, system (\ref{system})  can be reformulated as a --- singular --- perturbation of the cubic Szeg\"o equation (\ref{Szego}). Indeed, write $u_0=\e w_0$ and $u(t,x)=\e w(\e^2t, x-t)$, then $w=w_++w_-$ solves the system
\begin{equation}\label{perturb}\left\{\begin{array}{ll}
i\partial_t w_+&=\Pi_+(|w|^2w)\\
i(\e^2\partial _t-2\partial_x)w_-&=\e^2 \Pi_-(|w|^2w)\end{array}\right.\end{equation}
Notice that, for $\e =0$ and $\Pi_+w_0=w_0$, the solution of this system is exactly the solution of (\ref{Szego}). It is therefore natural to
ask how much, for $\e >0$ small, the solution of system (\ref{perturb}) stays close to the solution of equation (\ref{Szego}). Since equation (\ref{Szego}) turns out to be completely integrable, this problem appears as a perturbation of a completely integrable infinite dimensional system.  There is a lot of literature on this subject (see e.g. the books  by Kuksin \cite{K}, Craig \cite{C} and Kappeler--P\"oschel \cite{KP} for the KAM theory). In the case of the 1D cubic NLS equation and of the modified KdV equation, with special initial data such as solitons or 2-solitons, we refer to recent papers by Holmer-Zworski (\cite{HZ1}, \cite{HZ2}), Holmer-Marzuola-Zworski (\cite{HMZ}), Holmer-Perelman-Zworski (\cite{HPZ}) and to references therein. Here we emphasize that our perturbation is more singular and that we deal with general Cauchy data.
\s
Finally, let us mention that the proof of Theorem \ref{Approx} is based on a Poincar{\'e}-Birkhoff normal form approach, similarly to \cite{B} and \cite{G} for instance. More specifically, we prove that equation (\ref{CauchySzego}) turns out to be a Poincar{\'e}-Birkhoff normal form of equation (\ref{(W)}), for small initial data with only non negative modes.

As a corollary of Theorem \ref{Approx}, we get the following instability result.
\begin{corollary}\label{coro}
Let $s>1$. There exists a sequence of data $u_0^n$ and a sequence of times $t^{(n)}$
such that, for any $r$,
$$\Vert u_0^n\Vert _{H^r}\rightarrow 0$$
while the corresponding solution of (\ref{(W)}) satisfies
$$\Vert u^n(t^{(n)})\Vert _{H^s}\simeq \Vert u_0^n\Vert _{H^s}\, \left (\log \frac 1{\Vert u_0^n\Vert _{H^s}}\right )^{2s-1}\ .$$
\end{corollary}

It is interesting to compare this result to what is known about cubic NLS. 
In the one dimensional case, the cubic NLS is integrable \cite{ZS} and admits an infinite number of conservation laws which control the regularity of the solution in Sobolev spaces. As a consequence, no such norm inflation occurs.  This is in contrast with the 
2D cubic NLS case for which  Colliander, Keel, Staffilani, Takaoka, Tao exhibited in \cite{CKST} small initial data in $H^s$ which give rise to large $H^s$ solutions after a large time.

In our case, the situation is different. Although the cubic Szeg\"o equation is completely integrable, its conservation laws do not control the regularity of the solutions, which allows a large time behavior similar to the one proved in \cite{CKST} for 2D cubic NLS (see \cite{GG} section 6, corollary 5). Unfortunately, the time interval on which the approximation (\ref{approximation}) holds does not allow to infer large solutions for (\ref{(W)}), but only solutions with large relative size with respect to their Cauchy data --see section \ref{Corollary} below. A time interval of the form $[0,\frac 1{\e^{2+\beta}}]$ for some $\beta>0$ would be enough to construct large solutions for (\ref{(W)}) for some $H^s$-norms.
\s 
We close this introduction by mentioning that O. Pocovnicu solved a similar problem for equation (\ref{(W)}) on the line by using the renormalization group method instead of the Poincar\'e-Birkhoff normal form method. Moreover, she improved  the approximation in Theorem \ref{Approx} by introducing a quintic correction to the Szeg\"o cubic equation \cite{P}.
\s
The paper is organized as follows. In section \ref{LaxPair} we recall some basic facts about the Lax pair structure for the cubic Szeg\" o equation (\ref{Szego}). In section \ref{Corollary}, we deduce Corollary \ref{coro}  from Theorem \ref{Approx}.
Finally, the proof of Theorem \ref{Approx} is given in section 4. 
\s

\section{The Lax pair for the cubic Szeg\"o equation and some of its consequences}\label{LaxPair}
In this section, we recall some basic facts about equation (\ref{Szego}) (see \cite{GG} for more detail).
Given $w\in H^{1/2}(\T)$, we define (see {\it e.g.} Peller
\cite{Pe}, Nikolskii \cite{N}), the Hankel operator of symbol $w$ by
$$H_w(h)=\Pi_+ (w\overline h)\ ,\ h\in L^2_+\ .$$
It is easy to check that $H_w$ is a $\C $ -antilinear  Hilbert-Schmidt operator.
 In \cite{GG}, we proved that the cubic Szeg\"o flow admits a Lax pair in the following sense. For simplicity let us restrict ourselves to the case of $H^s$ solutions of (\ref{Szego}) for $s>\frac 12$. From \cite{GG} Theorem 3.1, there exists  a mapping $w\in H^s\mapsto B_w$, valued into $\C $-linear bounded skew--symmetric
operators on $L^2_+$, such that 
\begin{equation}\label{laxpair}
H_{-i\Pi_+(|w|^2w)}=[B_w,H_w]\ .
\end{equation}
Moreover, $$B_w=\frac i2 H^2_w-iT_{|w|^2}\ ,$$ where $T_b$ denotes the Toeplitz operator of symbol $b$ given by $T_b(h)=\Pi_+(bh)$.
Consequently, $w$ is a solution of (\ref{Szego}) if
and only if
\begin{equation}\label{laxpair2}
\frac d{dt} H_w=[B_w,H_w]\ .
\end{equation}

An important consequence of this structure  is that the cubic Szeg\"o equation admits an infinite number of conservation laws. Indeed, denoting $W(t)$ the solution of the operator equation
$$\frac {d}{dt}W=B_wW\ ,\ W(0)=I\ ,$$ the operator $W(t)$ is unitary for every $t$, and 
$$W(t)^*H_{w(t)}W(t)=H_{w(0)}.$$ 
Hence, if $w$ is a solution of (\ref{Szego}),
then $H_{w(t)}$ is unitarily equivalent to $H_{w(0)}$. Consequently, the spectrum of 
the $\C $-linear positive self adjoint trace class operator $H_w^2$ is conserved by the evolution. 
%Moreover, one can prove that 
%\begin{equation}\label{Bu}
%B_w=-iT_{\vert w\vert ^2}+\frac i2H_w^2\ ,
%\end{equation}
%where $T_b$ denotes the Toeplitz operator of symbol $b$,
%$$T_b(h)=\Pi (bh)\ .$$
In particular, the trace norm of $H_w$ is conserved by the flow. A theorem by Peller, see \cite{Pe}, Theorem 2, p. 454, states that the trace norm of a Hankel operator $H_w$ is equivalent to the norm of $w$ in the Besov space $B^1_{1,1}(\T )$.
Recall that the Besov space $B^1=B^1_{1,1}(\T )$ is  defined as the set of functions $w$ so that $\Vert w\Vert_{B^1_{1,1}}$ is finite where $$
\Vert w\Vert_{B^1_{1,1}}=\Vert S_0(w)\Vert_{L^1}+\sum_{j=0}^\infty 2^j\Vert \Delta_j w\Vert_{L^1},$$
here $w=S_0(w)+\sum_{j=0}^\infty \Delta_j w$ stands for the Littlewood-Paley decomposition of $w$. It is standard that $B^1$ is an algebra included into $L^\infty $ (in fact into the Wiener algebra). The conservation of the trace norm of $H_w$ therefore provides an $L^\infty $ estimate for solutions of (\ref{Szego}) with initial data in $B^1$.
\s
The space $B^1$ and formula (\ref{laxpair}) will play an important role in the proof of Theorem \ref{Approx}. In particular, the last part  will follow from the fact that $\Vert u(t)\Vert_{B^1}$ remains bounded by $\e$ for $t<\!\!<\frac 1{\e^3}$.
The fact that $H^s(\T)\subset B^1$ for $s>1$, explains why we assume $s>1$ in the statement.

\section{Proof of Corollary \ref{coro}}\label{Corollary}

As observed in \cite{GG}, section 6.1, Proposition 7, and section 6.2, Corollary 5, the equation 
$$
i\partial _tw=\Pi _+(\vert w\vert ^2w)\  ,\ w(0,x)=\frac{a_0\, {\rm e}^{ix}+b_0}{1-p_0{\rm e}^{ix}}$$
with $a_0,b_0,p_0\in \C , \vert p_0\vert <1$
can be  solved as
$$w(t,x)=\frac{a(t)\, {\rm e}^{ix}+b(t)}{1-p(t){\rm e}^{ix}}$$
where $a,b,p$ satisfy an ODE system explicitly solvable.

In the particular case when
$$a_0=\e \ ,\ b_0=\e \delta \ ,\ p_0=0\ ,\ w_\e(0,x)=\e ({\rm e}^{ix}+\delta )\  ,$$
one finds
$$1-\left \vert p \left (\frac \pi {2\e ^2\delta }\right )\right \vert ^2\simeq \delta ^2\ ,$$
so that, for $s>\frac 12$,
$$\left \Vert w_\e\left (\frac \pi {2\e ^2\delta }\right )\right \Vert _{H^s}\simeq \frac \varepsilon {\delta ^{2s-1}}\ .$$

Let $v_\e$ be the solution of $$
i(\partial _t+\partial _x)v_\e=\Pi _+(\vert v_\e\vert ^2v_\e)\  ,\ v_\e(0,x)=\e ({\rm e}^{ix}+\delta )$$
then $v_\e(t,x)=w_\e (t,x-t)$ so that 
$$\left \Vert v_\e\left (\frac \pi {2\e ^2\delta }\right )\right \Vert _{H^s}\simeq \frac \varepsilon {\delta ^{2s-1}}\ .$$
Choose
 $$\varepsilon=\frac 1n\ ,\ \delta =\frac {C}{\log n}\ $$ 
 with $C$ large enough so that if $t^{(n)}:=\frac \pi {2\e ^2\delta }$ then $t^{(n)}<c \frac{\log (1/\e )}{\e ^2}$, where $c=c_\alpha$ in Theorem \ref{Approx} for $\alpha=1$, say.
 Denote by $u_0^n:=v_\e(0,\cdot)$. 
 As $\Vert u_0^n\Vert_{H^s}\simeq \e$, the previous estimate reads 
 $$\left \Vert v_\e\left (\frac \pi {2\e ^2\delta }\right )\right \Vert _{H^s}\simeq \Vert u_0^n\Vert _{H^s}\, \left (\log \frac 1{\Vert u_0^n\Vert _{H^s}}\right )^{2s-1}\ .$$
 Applying Theorem \ref{Approx}, we get the same information about $\Vert u_n(t^{(n)})\Vert_{H^s}$.
 \s

\section{Proof of Theorem \ref{Approx}}
First of all, we rescale $u$ as $\e u$ so that equation (\ref{(W)}) becomes
\begin{equation}\label{Weps}\quad i\partial _t u-|D|u=\e ^2|u|^2u, \; u(0,\cdot)=u_0\end{equation}
with $\Vert u_0\Vert _{H^s}=1 $.

\subsection{Study of the resonances. }
We write the Duhamel formula as 
$$u(t)={\rm e}^{-it\vert D\vert }\underline u(t)$$
with
$$\hat {\underline u}(t,k)=\hat u_0(k)-i\e^2\sum _{k_1-k_2+k_3-k=0}I(k_1,k_2,k_3,k),$$
where
$$I(k_1,k_2,k_3,k)=\int _0^t{\rm e}^{-i\tau \Phi (k_1,k_2,k_3,k)}\hat {\underline u}(\tau ,k_1)\overline {\hat {\underline u}(\tau ,k_2)}\hat {\underline u}(\tau ,k_3)\, d\tau\ ,$$
and $$\Phi (k_1,k_2,k_3,k_4):=\vert k_1\vert -\vert k_2\vert+\vert k_3\vert-\vert k_4\vert \ .$$

If $\Phi(k_1,k_2,k_3,k_4)\neq 0$, an integration by parts in $I(k_1,k_2,k_3,k_4)$ provides an extra factor $\e^2$, hence the set of $(k_1,k_2,k_3,k_4)$ such that $\Phi(k_1,k_2,k_3,k_4)=0$  is expected to play a crucial role in the analysis. This set is described in the following lemma.
\begin{lemma}
Given ${(k_1,k_2,k_3,k_4)\in \Z ^4}$, 
$${k_1-k_2+k_3-k_4=0\  {\rm and}\ \vert k_1\vert -\vert k_2\vert+\vert k_3\vert-\vert k_4\vert=0}$$
if and only if at least one of the following properties holds :
\begin{enumerate}
\item ${\forall j, k_j\ge 0\ .}$
\item ${\forall j, k_j\le 0\ .}$
\item ${k_1=k_2\ ,\ k_3=k_4\ .}$
\item ${k_1=k_4\ ,\ k_3=k_2\ .}$
\end{enumerate}
\end{lemma}
\begin{proof}
Consider ${(k_1,k_2,k_3,k_4)\in \Z ^4}$ such that  $k_1-k_2+k_3-k_4=0$, $\vert k_1\vert -\vert k_2\vert+\vert k_3\vert-\vert k_4\vert=0$, and the $k_j$'s are not all non negative or all non positive. Let us prove in that case that either $k_1=k_2$ and $k_3=k_4$,
or $k_1=k_4$ and $ k_3=k_2\ .$
Without loss of generality, we can assume that at least one of the $k_j$ is positive, for instance $k_1$. Then, substracting both equations, we get that $|k_3|-k_3=|k_2|-k_2+|k_4|-k_4$. If $k_3$ is non negative, then, necessarily both $k_2$ and $k_4$ are non negative and hence all the $k_j$'s are non negative. Assume now that $k_3$ is negative. At least one among $k_2$, $k_4$ is negative. If both of them are negative, then $k_3=k_2+k_4$ but this would imply $k_1=0$ which is impossible by assumption. So we get either that $k_3=k_2$ (and so $k_1=k_4$) or $k_3=k_4$ (and so $k_1=k_2$). This completes the proof of the lemma.
\end{proof}

\subsection{First reduction}
We get rid of the resonances corresponding to cases (3) and (4) by applying the transformation
 \begin{equation}
 \label{transformation}
 u(t)\mapsto {\rm e}^{2it\e ^2\Vert u_0\Vert _{L^2}^2}u(t)\end{equation} which,  since the $L^2$ norm of $u$ is conserved,  leads to the equation
 
\begin{equation}\label{(W')}\quad i\partial _t u-|D|u=\e ^2(\vert u\vert ^2-2\Vert u\Vert_{L^2}^2)u, \; u(0,\cdot)=u_0\ .
\end{equation}
Notice that this transformation does not change the $H^s$ norm.
The Hamiltonian function associated to the equation (\ref{(W')}) is given by
$$ H(u)=\frac 12(|D|u,u)+ \frac {\e ^2}4(\Vert u\Vert_{L^4}^4- 2\Vert u\Vert _{L^2}^4)=H_0(u)+\e ^2R(u)\ ,$$
where $$H_0(u):=\frac 12(|D|u,u)\ ,\ R(u):=\frac 14(\Vert u\Vert_{L^4}^4- 2\Vert u\Vert _{L^2}^4)=\frac 14\sum_{k_1-k_2+k_3-k_4=0 ,\atop k_1\neq k_2,k_4}u_{k_1}\overline{u_{k_2}}u_{k_3}\overline{u_{k_4}}.$$
\subsection{The Poincar{\'e}-Birkhoff normal form}
We claim that under a suitable canonical transformation on $u$, $H$ can be reduced  to the following Hamiltonian
$$\tilde H(u)=H_0(u)+\e ^2\tilde R(u)+O(\e ^4)$$
where 

$$\tilde R(u) =\frac 14\sum _{{\bf k}\in \mathcal R}u_{k_1}\overline{u_{k_2}} u_{k_3}\overline{u_{k_4}}$$
with
\begin{eqnarray*}\mathcal R&=&\{ {\bf k}=(k_1,k_2,k_3,k_4): k_1-k_2+k_3-k_4=0,\\ & & k_1\ne k_2, k_1\ne k_4,\quad\forall j,\  k_j\ge 0\ \; {\rm or}\  \; \forall j,\ k_j\le 0\} \ .
\end{eqnarray*}

We look for a canonical transformation as the value at time $1$ of some Hamiltonian flow. In other words, we look for a function $F$ such that  its  Hamiltonian vector field is  smooth on $H^s$  and on $B^1$ , so that our canonical transformation is $\varphi_1$, where $\varphi_\sigma $ is the solution of 
\begin{equation}\label{flot}
\frac d{d\sigma} \varphi_\sigma (u)=\e ^2\, X_F(\varphi_\sigma (u)),\; \varphi_0(u)=u.
\end{equation}
Recall that, given a smooth real valued function $F$, its 
Hamiltonian vector field $X_F$ is defined by
$$dF(u).h=: \omega (h, X_F(u))\ ,$$
and, given two functions $F, G$ admitting
Hamiltonian vector fields, the Poisson bracket of $F, G$ is defined by
$$\{ F,G\}(u)=\omega (X_F(u), X_G(u))\ .$$

Let us make some preliminary remarks about the Poisson brackets.

In view of the expression of $\omega $, we have
  $$\{F,G\}:=dG.X_F=\frac 2{i}\sum_k (\partial_{\overline k}F\partial_kG- \partial_{\overline k}G\partial_k F)$$ where $\partial_kF$ stands for $\frac{\partial F}{\partial u_k}$ and $\partial_{\overline k}F$ for  $\frac{\partial F}{\partial \overline{u_k}}$. In particular, if $F$ and $G$ are respectively homogeneous of order $p$ and $q$, then their Poisson bracket is homogeneous of order $p+q-2$.
\s
We prove the following lemma.

% Denote by $K$ the subset of $M$ defined by
% $$|k_1|-|k_2|+|k_3|-|k_4|\neq 0\ ,$$
% and by $L$ the subset of $M$ where the $k_j$ have the same sign (nonnegative or nonpositive).
% Set
% $$\tilde R(u)=\frac 14\sum _{{\bf k} \in L}u_{k_1}\overline{u_{k_2}}u_{k_3}\overline{u_{k_4}}\ .$$
\begin{lemma}\label{X_F}
  Set $$F(u):=\sum_{k_1-k_2+k_3-k_4=0}f_{k_,k_2,k_3,k_4} u_{k_1}\overline{u_{k_2}}u_{k_3}\overline{u_{k_4}}\ ,$$
  where
  \begin{eqnarray*}f_{k_1,k_2,k_3,k_4}=\left\{\begin{array}{cc}
\frac {i}{4(|k_1|-|k_2|+|k_3|-|k_4|)} \text{ if }|k_1|-|k_2|+|k_3|-|k_4| \neq 0\ ,\\
0\text{ otherwise.}
\end{array}\right.
\end{eqnarray*}
Then  $X_F$ is smooth on $H^s, s>\frac 12,$ as well as  on $B^1$, and 
$$\{ F, H_0\} +R=\tilde R\ ,\ $$
$$\Vert DX_F(u)h\Vert\lesssim \Vert u\Vert ^2\Vert h\Vert\ ,$$ where the norm is taken either in $H^s, s>\frac 12,$ or in $B^1$.
\end{lemma}
\begin{proof}
First we make a formal calculation with $F$ given by $$F(u):=\sum_{k_1-k_2+k_3-k_4=0}f_{k_1,k_2,k_3,k_4} u_{k_1}\overline{u_{k_2}}u_{k_3}\overline{u_{k_4}}$$
for some coefficients $f_{k_1,k_2,k_3,k_4}$ to be determined later.
We compute 
$$\{ F,H_0\} =\frac 1i\sum _{k_1-k_2+k_3-k_4=0}(-\vert k_1\vert +\vert k_2\vert -\vert k_3\vert +\vert k_4\vert )f_{k_1,k_2,k_3,k_4}u_{k_1}\overline{u_{k_2}}u_{k_3}\overline{u_{k_4}}\ $$
so that equality $\{ F,H_0\} +R=\tilde R$ requires 
 \begin{eqnarray*}f_{k_1,k_2,k_3,k_4}=\left\{\begin{array}{clcl}
&\frac {i}{4(|k_1|-|k_2|+|k_3|-|k_4|)} &\text{ if }|k_1|-|k_2|+|k_3|-|k_4| \neq 0\\
&0&\text{ otherwise.}
\end{array}\right.
\end{eqnarray*}

One can easily check that the function $F$ is explicitly given by 
$$F(u)= \frac 12 {\rm Im}\left( (D_0^{-1}u_-\vert |u_+|^2u_+)-(D_0^{-1}u_+\vert |u_-|^2u_-)-(D_0^{-1} |u_+|^2\vert |u_-|^2)\right)$$
where $D_0^{-1}$ is the operator defined by
$$D_0^{-1}u(x)=\sum_{k\neq 0}\frac{ u_k}k\, {\rm e}^{ikx}.$$
In view of the above formula, the Hamiltonian vector field $X_F(u)$ is a sum and products of terms involving the following maps
$f\mapsto \overline f$, $f\mapsto D_0^{-1}f$, $f\mapsto \Pi_\pm f$, $(f,g)\mapsto fg$. These maps are continuous on $H^s$ and on $B^1$. Hence, $X_F$ is smooth and its differential satisfies the claimed estimate on $H^s, s>\frac 12,$ and $B^1$.
\end{proof}
For further reference, we state the following technical lemma, which is based on straightforward calculations.
 \begin{lemma}\label{PoissonFR}
 The function $\tilde R$ and its Hamiltonian vector field  are given by 
  \begin{eqnarray*}
 \tilde R(u)=\frac 14 ( \Vert \tilde u_+\Vert _{L^4}^4+\Vert \tilde u_-\Vert  _{L^4}^4)+{\rm Re}((u\vert 1)\, ( u_-\vert  u_- ^2))-\frac 12(\Vert u_+\Vert _{L^2}^4 +\Vert u_-\Vert _{L^2}^4)\  ,\\
   iX_{\tilde R}(u)=\Pi _+(\vert u_+\vert ^2u_+)+\Pi _-(\vert u_-\vert ^2u_-)-2\Vert u_+\Vert ^2_{L^2}u_+ -2\Vert u_-\Vert _{L^2}^2u_-\hskip1.5cm
    \\ \qquad +(u_-\vert u_-^2)
  +2(1\vert u)\Pi _-(\vert u_-\vert ^2)+(1\vert u)u_-^2\ ,
  \end{eqnarray*}
    where we have set $u_\pm := \Pi _{\pm }(u)$.\\
The maps $X_{\{ F,R\} }$ and $X_{\{ F,\tilde R\} }$ are  smooth homogeneous polynomials of degree five on $B^1$ and on $H^s$ for every $s>\frac 12$.
\end{lemma}
We now perform the canonical transformation
$$\chi_\e :={\rm exp}(\e ^2X_F).$$

\begin{lemma}\label{phi_sigma}
Set $\varphi_\sigma:={\rm exp}(\e^2\sigma X_F)$ for $-1\le \sigma \le 1$. There exist $m_0>0$ and $C_0>0$ so that, for any $u\in B^1$ so that $\e\Vert u\Vert_{B^1}\le m_0$, $\varphi _\sigma (u)$ is well defined for $\sigma \in [-1,1]$ and 
$$\Vert \varphi_\sigma(u)\Vert_{B^1}\le \frac 32\Vert u\Vert_{B^1}$$
$$\Vert \varphi_\sigma(u)-u\Vert_{B^1}\le C_0\e^2\Vert u\Vert_{B^1}^3$$
$$\Vert D\varphi_\sigma(u)\Vert_{B^1\to B^1}\le e^{C_0\e^2\Vert u\Vert_{B^1}^2}$$
Moreover, the same estimates hold in $H^s, s>\frac 12, $ with some constants $m(s)$ and $C(s)$.
\end{lemma}

\begin{proof}
Write $\varphi_\sigma$ as the integral of its derivative and use Lemma \ref{X_F} to get
\begin{equation}\label{EquPhi_sigma}\sup_{|\sigma|\le \tau }\Vert \varphi_\sigma (u)\Vert_{B^1}\le \Vert u\Vert_{B^1}+C\e^2 \sup_{|\sigma|\le \tau }\Vert \varphi_\sigma (u)\Vert_{B^1}^3\ ,\ 0\le \tau \leq 1
\end{equation}
We now use the following standard bootstrap lemma.
\begin{lemma}\label{bootstrap}
Let $a,b, T>0$ and $\tau \in [0,T] \mapsto M(\tau )\in \R _+$ be a continuous function satisfying
$$\forall \tau \in [0,T], M(\tau )\leq a+bM(\tau )^3\ .$$
Assume 
$$\sqrt {3b}\, M(0)<1\ ,\ \sqrt{3b}\, a<\frac 23\ .$$
Then
$$\forall \tau \in [0,T]\ ,\  M(\tau )\le \frac 32 a\ .$$
\end{lemma}
\begin{proof}For the convenience of the reader, we give the proof of Lemma \ref{bootstrap}. The function $f:z\ge 0\mapsto z-bz^3$ attains its  maximum at $z_c=\frac 1{\sqrt{3b}}$, equal to $f_m=\frac 2{3\sqrt {3b}}$. 
Consequently, since $a$ is smaller than $f_m$ by the second inequality, 
$$\{ z\ge 0 : f(z)\le a\} =[0, z_-]\cup [z_+,+\infty )$$ 
with $z_-<z_c<z_+$ and $f(z_-)=a$. Since $M(\tau )$ belongs to this set for every $\tau $ and since $M(0)$ belongs to the first interval by the first inequality, we conclude by continuity 
that $M(\tau )\le z_-$ for every $\tau $. By concavity of $f$, $f(z)\ge \frac 23 z$ for $z\in [0,z_c]$, hence $z_-\le \frac 32 a$.
\end{proof}
Let us come back to the proof of Lemma \ref{phi_sigma}. If $\varepsilon \Vert u\Vert_{B^1}< \frac 2{3\sqrt {3C}}$,   equation (\ref{EquPhi_sigma}) and Lemma \ref{bootstrap} imply  that
\begin{equation}\label{bornephi}
\sup_{|\sigma|\le 1}\Vert \varphi_\sigma (u)\Vert_{B^1}\le \frac 32\Vert u\Vert_{B^1}\ ,
\end{equation}
which is  the first estimate.
For the second one, we write for $|\sigma|\le 1$, 
$$\Vert \varphi_\sigma(u)-u\Vert_{B^1}=\Vert \varphi_\sigma(u)-\varphi_0(u)\Vert_{B^1}\le |\sigma|\sup_{|s|\le |\sigma|}\left\Vert \frac d{ds}\varphi_s(u)\right\Vert_{B^1}\le C_0 \e^2\Vert u\Vert_{B^1}^3,$$
where the last inequality comes from Lemma \ref{X_F} and estimate (\ref{bornephi}). 
\s
It remains to prove the last estimate. We differentiate the equation satisfied by $\varphi_\sigma$ and use again Lemma \ref{X_F} to obtain 
\begin{eqnarray*}\Vert D\varphi_\sigma(u)\Vert_{B^1\to B^1}&\le&1+\e^2\left|\int_0^{\sigma}\Vert DX_F(\varphi_\tau(u))\Vert_{B^1\to B^1} \Vert D\varphi_{\tau}(u)\Vert_{B^1\to B^1}\, d\tau\right|\\
&\le&1+C_0\e^2\Vert u\Vert_{B^1}^2\left|\int_0^{\sigma}\Vert D\varphi_{\tau}(u)\Vert_{B^1\to B^1}d\tau\right |,
\end{eqnarray*}
and Gronwall's lemma yields  the result.
Analogous proofs give the estimates in $H^s$.
\end{proof}

Let $u$ satisfy the assumption of  Lemma \ref{phi_sigma} in $B^1$ or in $H^s$ for some $s>\frac 12$.

Let us compute $H\circ \chi_\e=H\circ \varphi _1$ as the Taylor expansion of $H\circ \varphi _\sigma $ at time $1$ around $0$. One gets
\begin{eqnarray*}
H\circ\chi_\e&=&H\circ\varphi _1=H_0\circ \varphi _1+\e ^2R\circ \varphi _1\\
&=&H_0+\frac d{d\sigma }[H_0\circ\varphi _\sigma ]_{\sigma =0}+\e ^2R+\\
&+& \int _0^1\left ((1-\sigma )\frac { d^2}{d\sigma ^2}[H_0\circ\varphi _\sigma ]+\e ^2\frac d{d\sigma}[R\circ \varphi _\sigma ]\right )\, d\sigma \\
&=&H_0+\e ^2(\{F,H_0\}+R)+\e ^4\int _0^1\left ((1-\sigma )\{F,\{F, H_0\}\} +\{ F,R\} \right ) \circ \varphi _\sigma \, d\sigma \\
&=&H_0+\e ^2\tilde R+\e ^4\int _0^1\left ((1-\sigma )\{F,\tilde R\} +\sigma\{ F,R\} \right ) \circ \varphi _\sigma \, d\sigma \\
&:=&H_0+\e ^2\tilde R+\e ^4\int _0^1G(\sigma )\circ \varphi _\sigma \, d\sigma \ .
\end{eqnarray*}
By Lemma \ref{PoissonFR},  one gets
$$\sup _{0\le \sigma \le 1}\Vert X_{G(\sigma )}(w)\Vert \le C\Vert w\Vert ^5$$
where the norm stands for the $B^1$ norm or the $H^s$ norm. Since
$$X_{G(\sigma )\circ \varphi _\sigma }(u)=D\varphi _{-\sigma }(\varphi _\sigma (u)).X_{G(\sigma )}(\varphi _\sigma (u))\ ,$$
we conclude from Lemma \ref{phi_sigma}  that, if $\e \Vert u\Vert _{B^1}\le m_0$, 
$$\Vert X_{G(\sigma )\circ \varphi _\sigma }(u)\Vert _{B^1}\leq C\Vert u\Vert _{B^1}^5\ .$$
As a consequence, one can write 
$$X_{H\circ \chi_\e}=X_{H_0}+\e ^2X_{\tilde R}+\e ^4Y\ ,$$
where, if $\e \Vert u\Vert _{B^1}\le m_0$, 
then $$\Vert Y(u)\Vert _{B^1}\lesssim \Vert u\Vert ^5_{B^1}Ê\ .$$
An analogous estimate holds in $H^s$, $s>\frac 12$.

\s
\subsection{End of the proof} 
We first deal with the $B^1$-norm of $u$, solution of equation (\ref{(W')}). We are going to prove that $\Vert  u(t)\Vert_{B^1}=\mathcal O(1)$ for $t<\!\!<\frac 1{\e^3}$ by the following  bootstrap argument.
We assume that for some $K$ large enough with respect to $\Vert u_0\Vert _{B^1}$, for some $T>0$,  for all $t\in [0,T]$, $\Vert  u(t)\Vert _{B^1}\le 10 K$, and we prove that if $T<\!\!<\frac 1{\e^3}$,  $\Vert u(t)\Vert _{B^1}\le K$ for $t\in [0,T]$. 
This will prove the result by continuity.
 \s
 Set, for $t\in [0,T]$, 
 $$\tilde u(t):=\chi_\e ^{-1}(u(t))\ ,$$
 so that $\tilde u$ is solution of
 $$i\partial _t\tilde u-\vert D\vert \tilde u=\e ^2X_{\tilde R}(\tilde u)+\e ^4Y(\tilde u)\ .$$
 Moreover, by  Lemma \ref{phi_sigma},
 $$\Vert \tilde u(t)-u(t)\Vert _{B^1}\lesssim \e^2\Vert u\Vert_{B^1}^3$$ and so by the hypothesis, $\Vert \tilde u(t)\Vert _{B^1}\le 11K$ if $\e $ is small enough.
In view of the expression of the Hamiltonian vector field of $\tilde R$ in Lemma \ref{PoissonFR}, the equation  
 for $\tilde u$ reads
$$\left\{\begin{array}{ccl}
i\partial _t\tilde u_+-D\tilde u_+&=&\e ^2\left (\Pi_+ (\vert \tilde u_+\vert ^2\tilde u_+)-2\Vert \tilde u_+\Vert _{L^2}^2\tilde u_++ \int _{\T}\vert \tilde u_-\vert ^2\tilde u_-\right )\\
&+&\e ^4Y_+(\tilde u)\ ,\\
 \\
i\partial _t\tilde u_-+D\tilde u_-&=&\e ^2\left (\Pi_- (\vert \tilde u_-\vert ^2\tilde u_-)-2\Vert \tilde u_-\Vert _{L^2}^2\tilde u_- +2(1\vert \tilde u)\Pi _-(\vert \tilde u_-\vert ^2)+(1\vert \tilde u)\tilde u_-^2\right )\\
&+&\e ^4Y_-(\tilde u)\  .\end{array}\right.
$$
 Notice that all the Hamiltonian functions we have dealt with so far are invariant by multiplication by complex numbers of modulus $1$, hence their Hamiltonian vector fields satisfy
 $$X({\rm e}^{i\theta}z)={\rm e}^{i\theta }z\ ,$$
 so that the corresponding Hamiltonian flows conserve the $L^2$ norm. Hence $\tilde u$ has the same $L^2$ norm as $u$, which is the $L^2$ norm of $u_0$. In particular, $\vert (1\vert \tilde u)\vert \le \Vert  u_0\Vert_{L^2}$.

 Moreover, as $\Vert u_0\Vert_{B^1} \lesssim \Vert u_0\Vert_{H^s}=\mathcal O(1)$ since $s>1$, $\tilde u_0$ satisfies
 $$\Vert \tilde u_0-u_0\Vert _{B^1}\lesssim \e ^2\ $$ by Lemma \ref{phi_sigma} so that, as
 $u_{0-}=0$, we get $\Vert\tilde {u_0}_-\Vert_{B^1}= \mathcal O(\e^2)$.
 Then we obtain from the second equation
 $$\sup _{0\le \tau \le t}\Vert \tilde u_-(\tau )\Vert _{B^1}\lesssim \e ^2+ \e ^2t(\sup _{0\le \tau \le t}\Vert \tilde u_-(\tau )\Vert _{B^1}^3+\sup _{0\le \tau \le t}\Vert \tilde u_-(\tau )\Vert _{B^1}^2)+\e ^4 tK^5\ .$$
 
 Let $M(t)=\frac 1\e\sup _{0\le \tau \le t}\Vert \tilde u_-(\tau )\Vert _{B^1}$ so that, if $t\le T$, $$M(t)\lesssim \e+\e^3 TM(t)^2(1+\e M(t))+\e^3T.$$
 As $3m^2\le 1+2m^3$ for any $m\ge 0$, we get
 $$M(t)\lesssim \e+\e^3TM(t)^3+\e^3 T.$$
 Using Lemma \ref{bootstrap}, we conclude that, if $T<\!\!<\frac 1{\e^3}$, 
  $$\sup _{0\le \tau \le T}\Vert \tilde u_-(\tau )\Vert _{B^1}<\!\!< \e .$$
For further reference, notice that, if $T\lesssim \frac 1{\e ^2}\log \frac 1{\e }$, this estimate can be improved as
 $$\sup _{0\le \tau \le T}\Vert \tilde u_-(\tau )\Vert _{B^1}\lesssim \e ^{2-\alpha }\ ,\  \forall \alpha >0.$$

We come back to the case $T<\!\!<\frac 1{\e^3}$. From the estimate on $\tilde u_-$, we infer
 $$\Vert \tilde u_+\Vert _{L^2}^2=\Vert \tilde u\Vert _{L^2}^2+ O(\e^2 )=\Vert u_0\Vert _{L^2}^2+ O(\e^2 )\ ,$$
  and the equation for $\tilde u_+$ reads 
  $$i\partial _t\tilde u_+-D\tilde u_+=\e ^2\left (\Pi_+ (\vert \tilde u_+\vert ^2\tilde u_+)-2\Vert u_0\Vert _{L^2}^2\tilde u_+\right )+\e ^4Y_+(\tilde u)+\mathcal O(\e ^5)+\mathcal O(\e ^4)\tilde u_+\ .$$
 Since $\tilde{u_0}_+$ is not small in $B^1$, we have to use a different strategy to estimate $\tilde u_+$. We  use the complete integrability of the cubic Szeg\"o equation, especially its Lax pair and the conservation of the $B^1$-norm.  
 \s
  At this stage it is of course convenient to cancel the linear term $\Vert u_0\Vert _{L^2}^2 \tilde u_+$ by multiplying $\tilde u_+(t)$ by ${\rm e}^{2i\e ^2t\Vert u_0\Vert _{L^2}^2}$. As pointed out before, this change of unknown is completely transparent to the above system. This leads to
 $$ i\partial _t\tilde u_+-D\tilde u_+=\e ^2\Pi_+ (\vert \tilde u_+\vert ^2\tilde u_+)+\e ^4Y_+(\tilde u)+\mathcal O(\e ^5)+ \mathcal O(\e ^4)\tilde u_+\ .$$

We now appeal to the results recalled in section \ref{LaxPair}. We introduce
the unitary family $U(t)$ defined by
$$i\partial_t U-DU=\e ^2(T_{\vert \tilde u_+\vert ^2}-\frac 12H_{\tilde u_+}^2)U\ ,\ U(0)=I,$$
so that, using formula (\ref{laxpair}),
$$i\partial _t(U(t)^*H_{\tilde u_+(t)}U(t))=\e ^{4}U(t)^*H_{Y_+(\tilde u)+\mathcal O(\e)+\mathcal O(1)\tilde u_+}U(t)\ .$$
 Then, we use Peller's theorem \cite{Pe} which states, as recalled in section \ref{LaxPair}, that the trace norm of a Hankel operator of symbol $b$ is equivalent to the $B^1$-norm of $b$ to obtain 
 \begin{eqnarray*}
 \Vert {\tilde u_+(t)}\Vert_{B^1}&\simeq&{\rm Tr}  |H_{\tilde u_+(t)}|\\
 &\lesssim &{\rm Tr}   |H_{\tilde {u}_{0+}}|+\e ^4\int_0^t ({\rm Tr} |H_{Y_+(\tilde u)}(\tau)|+{\rm Tr } |H_{\tilde u_+}(\tau )|+\e )\, d\tau
\\
&\lesssim &  \Vert {\tilde{ u}_{0+}}\Vert_{B^1}+\e ^4\int_0^t (\Vert \tilde u(\tau)\Vert_{B^1}^5 + \Vert \tilde u_+(\tau )\Vert _{B^1}+\e )\, d\tau
\end{eqnarray*}
 so that as $\Vert \tilde u(t)\Vert_{B^1}\le 11K$,
 $$\Vert \tilde u_+(t)\Vert _{B^1}\lesssim \Vert {\tilde{ u_0}_+}\Vert_{B^1}+\e^4t(11K)^5\ ,$$
 and, if $t<\!\!<\frac 1{\e^3}$ and $\e $ is small enough,
 $$\Vert \tilde u(t)\Vert _{B^1}\lesssim \Vert {\tilde{ u_0}_+}\Vert_{B^1}\ .$$
 Using again the second estimate in Lemma \ref{phi_sigma}, we infer
 $$\Vert u(t)\Vert _{B^1}\leq K\ .$$
 Finally, using the inverse of transformation (\ref{transformation}) and multiplying $u$ by $\e$, we obtain estimate (\ref{estimateLinfty}) of Theorem \ref{Approx}.
\s
We now estimate the difference between the solution of the wave equation and the solution of the cubic Szeg\"o equation. Since we have applied transformation (\ref{transformation}), we have to compare  in $B^1$ the solution $u$ of equation (\ref{(W')}) to the solution $v$ of equation

$$i\partial _tv-Dv=\e ^2(\Pi _+(\vert v\vert ^2v)-2\Vert u_0\Vert_{L^2}^2v)\ ,\ v(0)=u _0\ .$$
Notice that, as $u_0$ is bounded in $H^s$, $s>1$,  and as the $B^1$ norm is conserved by the cubic Szeg\"o flow, 
$$\Vert v(t)\Vert_{B^1}\simeq \Vert u_0\Vert_{B^1}\lesssim \Vert u_0\Vert_{H^s}=\mathcal O(1).$$
We shall prove that, for every $\alpha >0$, there exists $c_\alpha >0$ such that, 
$$\forall t\leq \frac{c_\alpha }{\e ^2}\log \frac 1\e \ ,\ \Vert u(t)-v(t)\Vert _{B^1}\leq \e ^{2-\alpha }\ .$$
In view of the previous estimates, it is enough to prove that, on the same time interval,
$$\Vert \tilde u_+(t)-v(t)\Vert _{B^1}\leq \e ^{2-\alpha }\ ,$$
where $\tilde u_+$ satisfies
  \begin{equation}\label{eqtildeu+}
  \left\{\begin{array}{rcl}
  i\partial _t\tilde u_+-D\tilde u_+&=&\e ^2\left (\Pi_+ (\vert \tilde u_+\vert ^2\tilde u_+)-2\Vert u_0\Vert _{L^2}^2\tilde u_+\right )+\mathcal O(\e^4)\ ,\\
   \\
  \tilde u_+(0)&=&\tilde u_{0,+}\ .
  \end{array}\right.
  \end{equation}
As
 $\Vert \tilde u(t)\Vert _{B^1}\lesssim 1$, $\Vert v(t)\Vert _{B^1}\lesssim 1$, $\Vert \tilde u_{0,+}-u_0\Vert_{B^1}\le \e^2\Vert u_0\Vert_{B^1}\lesssim \e^2$
 and $$(i\partial_t-D)(\tilde u_+-v)=\e^2\Pi_+(|\tilde u_+|^2\tilde u_+-|v|^2v-2\Vert u_0\Vert_{L^2}^2(\tilde u_+-v))+\mathcal O(\e^4)\ ,$$
 we get, using that $B^1$ is an algebra on which $\Pi_+ $ acts, 
$$\Vert \tilde u_+(t)-v(t)\Vert _{B^1}\lesssim \e ^2+\e ^{4}t+\e ^2\int _0^t \Vert \tilde u_+(\tau )-v(\tau )\Vert _{B^1}\, d\tau \ .$$
This yields
$$\Vert \tilde u_+(t)-v(t)\Vert _{B^1}\lesssim (\e ^2+\e ^{4}t){\rm e}^{\e ^2t}\ ,$$
hence, for $t\le \frac {c_\alpha }{\e ^2}\log \frac 1\e $,
$$\Vert \tilde u_+(t)-v(t)\Vert _{B^1}\le \e ^{2-\alpha}\ .$$
\s
We now turn to the estimates in $H^s$ for $s>1$. 
\s
From the equation on $v$ and the a priori estimate in $B^1$, 
it follows that $\Vert v(t)\Vert _{H^s}\le Ae^{A\e^2t}$, $t>0$, so that $\Vert v(t)\Vert_{H^s}\le N(\e)$ for $t\le \frac {c}{\e^2}\log(\frac 1\e)$, $0<c<\!\!<1$ where $N(\e):=A\e^{-cA}$.
\s
Let us assume that for some $T>0$,
 $$\forall t\in [0,T], \Vert u(t)\Vert _{H^s}\le 10 N(\e)\, \ .$$
 We are going to prove that, for every $\alpha >0$, there exists $c_\alpha>0$ such that, if 
$$T\le \frac{c_\alpha}{\e ^2}\log \frac 1\e \ ,$$
then 
$$\forall t\in [0,T], \Vert u(t)-v(t)\Vert _{H^s}\le \e ^{2-\alpha }\ ,$$
Since $\Vert v(t)\Vert _{H^s}\le N(\e)$ for $t\le \frac{c}{\e ^2}\log \frac 1\e \ ,$ this will prove the result 
by a bootstrap argument. 
 \s
 As before, we perform the same canonical transformation
 $$\tilde u(t):=\chi_\e ^{-1}(u(t))\ ,$$
 to get the solution of
 $$i\partial _t\tilde u-\vert D\vert \tilde u=\e ^2X_{\tilde R}(\tilde u)+\e ^4Y(\tilde u)\ .$$
By Lemma \ref{phi_sigma},
 $$\Vert \tilde u(t)-u(t)\Vert _{H^s}\lesssim \e ^2N(\e )^3$$ and so $\Vert \tilde u(t)\Vert _{H^s}\lesssim  N(\e ).$ Therefore it suffices to prove that $$\forall t\in [0,T], \Vert \tilde u(t)-v(t)\Vert _{H^s}\le \e ^{2-\alpha }\ .$$
%We have also $$\Vert \tilde u(t)\Vert _{B^1_{1,1}}\lesssim\Vert \tilde u_0\Vert _{B^1_{1,1}} +\e ^4\int_0^t Y(\tilde u(s)) ds.$$
%In particular $\Vert \tilde u(t)\Vert _{B^1_{1,1}}=\mathcal O(1)$ for $t<<\frac 1{\e^3}$.
%As $X_F$ is smooth on $B^1_{1,1}$, it implies that $\Vert u(t)\Vert _{B^1_{1,1}}=\mathcal O(1)$ for $t<<\frac 1{\e^3}$.
 \s
 We first deal with $\tilde u_-$. A similar argument as the one developed in $B^1$ gives that for, for $0\le t\lesssim \frac 1{\e ^2}\log \frac 1\e $,
 $$\sup _{0\le \tau \le t}\Vert \tilde u_-(\tau )\Vert _{H^s}\le C_\alpha  \e ^{2-\alpha }$$
for every $\alpha >0$.  
% Then we perform the transformation
% $$\tilde u_+(t)\mapsto e^{-2i\e ^2t\Vert u_0\Vert _{L^2}^2}\tilde u_+(t)$$
% and we reduce to the following equation in $H^s$,
% $$
%  \left\{\begin{array}{rcl}
%  i\partial _t\tilde u_+-D\tilde u_+&=&\e ^2\Pi_+ (\vert \tilde u_+\vert ^2\tilde u_+)+\mathcal O(\e^4N(\e)^5)\ ,\\
%   \\
%  \tilde u_+(0)&=&\tilde u_{0,+}\ .
%  \end{array}\right.
%$$
 
It remains to estimate the $H^s$ norm of $\tilde u_+-v$. Notice that $$\Vert \tilde u_{0,+}-u_0\Vert_{H^s}\le \e^2$$ by Lemma \ref{phi_sigma}. We use the following inequality --- recall that $B^1\subset L^\infty $ ,
\begin{eqnarray*}
\Vert \Pi _+(\vert u\vert ^2u-\vert v\vert ^2v)\Vert _{H^s}&\lesssim &(\Vert u\Vert _{B^1}^2+\Vert v\Vert _{B^1}^2)\Vert u-v\Vert _{H^s}+\\
&+&(\Vert v\Vert _{H^s}+\Vert u-v\Vert _{H^s})(\Vert u\Vert _{B^1}+\Vert v\Vert _{B^1})\Vert u-v\Vert _{B^1}\ .
\end{eqnarray*}
 Plugging this into a Gronwall inequality, in view of  the previous estimates, we finally get
 $$\Vert \tilde u_+(t)-v(t)\Vert _{H^s}\le \e ^{2-\alpha}$$
 for $t\le \frac {c_\alpha }{\e ^2}\log \frac 1\e $. This completes the proof.
 
 \section{Appendix: a necessary condition for wellposedness}
 
 In this section, we justify that the boundedness in $H^s$ of the first iteration map of the Duhamel formula 
 $$F(t)=e^{-it|D|}f-i\int_0^t e^{-i(t-\tau)|D|} (|F(\tau)|^2F(\tau)) d\tau$$
implies $$\int_0^1\Vert e^{-it|D|}f\Vert_{L^4(\T)}^4 dt\lesssim \Vert f\Vert_{H^{s/2}}^4.$$
 
Indeed, assume the following inequality
$$\left \Vert \int_0^1 e^{-i(1-\tau)|D|} (|e^{-i\tau|D|}f|^2 e^{-i\tau|D|}f)d\tau\right\Vert _{H^s}\lesssim \Vert f\Vert_{H^s}^3.$$
 We compute the scalar product of the expression  in the left hand side with $e^{-i |D|}f$ and we get 
 $$\int_0^1\Vert e^{-i\tau |D|}f\Vert_{L^4}^4 d\tau \lesssim \Vert f\Vert_{H^s}^3\Vert f\Vert_{H^{-s}}.$$
 If we assume first that $f$ is spectrally supported, that is if $f=\Delta_N f$ for some $N$, then
 $\Vert f\Vert_{H^{\pm s}}\simeq N^{\pm s}\Vert f\Vert_{L^2}$ and the preceding inequality reads
 $$\int_0^1\Vert e^{-i\tau|D|}f\Vert_{L^4}^4 d\tau\lesssim N^{2s} \Vert f\Vert_{L^2}^4.$$
 
 Eventually, for general $f=\sum_{N}\Delta_N(f)$, we used the Littlewood-Paley estimate
 $$\Vert g\Vert_{L^4}^4\lesssim \sum_{N}\Vert \Delta_Ng\Vert^4_{L^4}$$
 to  get
 $$\int_0^1\Vert e^{-i\tau |D|}f\Vert_{L^4}^4 d\tau\lesssim  \Vert f\Vert_{H^{s/2}}^4.$$


\begin{thebibliography}{MTW1}
 
 \bibitem {B} Bambusi D., {\em Birkhoff normal form for some nonlinear PDEs,} Comm.
Math. Physics 234 (2003), 253--283.
 

\bibitem{CKST} Colliander  J., Keel M.,  Staffilani G.,Takaoka H., Tao, T., {\em Transfer of energy to high frequencies in the cubic defocusing nonlinear Schrodinger equation}, Inventiones Math.181 (2010), 39--113.

\bibitem{C} Craig, W., {\em Probl\`emes de petits diviseurs dans les \' equations aux d\' eriv\'ees partielles},  Panoramas et Synth\`eses, 9. Soci\'et\'e Math\'ematique de France, Paris, 2000. 

\bibitem{GGX} G\'erard, P., Grellier, S., {\em The cubic Szeg\"o equation} , S\'eminaire X-EDP,  20 octobre 2008, \'Ecole polytechnique, Palaiseau.

\bibitem{GG} G\'erard, P., Grellier, S., {\em The cubic Szeg\"o equation} , Ann. Scient. \'Ec. Norm. Sup. 43 (2010), 761-810.

\bibitem{GG2}G\'erard, P., Grellier, S., {\em Invariant Tori for the cubic  Szeg\"o equation},  to appear in Inventiones Mathematicae, DOI 10.1007/s00222-011-0342-7.

%\bibitem{GG3}G\'erard, P., Grellier, S., {\em Spectral inverse problems for compact Hankel operators: an explicit resolution}, in preparation.

\bibitem {G} Gr\' ebert, B., {\em Birkhoff Normal Form and Hamiltonian PDEs,}
arXiv math.AP (2006), 55 pages.

\bibitem{HZ1}Holmer, J., Zworski M., {\em Slow soliton interaction with delta impurities,} J. Mod.
Dyn. 1, no. 4 (2007), 689--718.

\bibitem{HZ2} Holmer, J., Zworski M., {\em  Soliton interaction with slowly varying potentials,} Int. Math. Res. Not., (2008), no. 10, Art. ID rnn026, 36 pp.

\bibitem{HMZ} Holmer, J., Marzuola J.,Zworski M., {\em Fast soliton scattering by delta impurities} Comm. Math. Phys. 274, Number 1 (2007), 187-216. 

\bibitem {HPZ} Holmer, J., Perelman, G.,  Zworski M., {\em Effective dynamics of double solitons for
perturbed mKdV,}  with an appendix by Bernd Sturmfels, Comm. Math. Phys. 305(3)(2011), 363--425. 

\bibitem{KP} Kappeler, T., P\"oschel, J.,  {\em KdV \& KAM. }Ergebnisse der Mathematik und ihrer Grenzgebiete. 3. Folge. A Series of Modern Surveys in Mathematics [Results in Mathematics and Related Areas. 3rd Series. A Series of Modern Surveys in Mathematics], 45. Springer-Verlag, Berlin, 2003. 
\bibitem {K} Kuksin, S., {\em  Nearly integrable infinite-dimensional Hamiltonian systems}. Lecture Notes in Mathematics, 1556. Springer-Verlag, Berlin, 1993. 

%\bibitem{Ha} Hartman, P., {\em On completely continuous Hankel matrices,} Proc. Amer. Math. Soc. 9 (1958), 862--866.
%\bibitem{Kr} Kronecker, L. : {\em Zur Theorie der Elimination einer
%Variablen aus zwei algebraische Gleischungen} Montasber. K\"onigl.
%Preussischen Akad. Wies. (Berlin), 535-600 (1881). Reprinted in {\em
%mathematische Werke}, vol. 2, 113--192, Chelsea, 1968.

%\bibitem{L} Lax, P. : {\em Integrals of Nonlinear equations of
%Evolution and Solitary Waves}, Comm. Pure and Applied Math. 21,
%467-490 (1968).

%\bibitem{L2} Lax, P. : {\em Periodic solutions of the the KdV equation.}  Comm. Pure Appl. Math.  28 , 141--188
%(1975).

%
%\bibitem{MPT} Megretskii, A~V., Peller, V.~V., and Treil, S.~R., The inverse problem for self-adjoint Hankel operators, Acta Math. 174 (1995), 241-309.

%\bibitem{Ne} Nehari, Z. : {\em On bounded bilinear forms.} Ann.
%Math. 65, 153--162 (1957).


\bibitem{N} Nikolskii, N.~K., {\em Operators, functions, and systems: an easy reading. Vol. 1. Hardy, Hankel, and
Toeplitz.} Translated from the French by Andreas Hartmann.
Mathematical Surveys and Monographs, 92. American Mathematical
Society, Providence, RI, 2002.

%\bibitem{N2} Nikolskii, N.~K. : 
%{\em Treatise on the shift operator.}
%Spectral function theory. With an appendix by S. V. Khrushch\"ev and V. V. Peller. Translated from the Russian by Jaak Peetre. Grundlehren der Mathematischen Wissenschaften [Fundamental Principles of Mathematical Sciences], 273. Springer-Verlag, Berlin, 1986.
\bibitem{P} Pocovnicu, O., {\em First and second order approximations for a non linear equation}, in {\sl Etude d'une \'equation non lin\'eaire, non dispersive et compl\`etement int\'egrable et de ses perturbations} PhD thesis,  universit\'e Paris-sud,  2011, and paper to appear.

\bibitem{Pe} Peller, V.~V., {\em Hankel operators of class
$\mathfrak{S}_p$ and their applications (rational approximation,
Gaussian processes, the problem of majorization of operators}, Math.
USSR Sb. 41, 443-479 (1982).

%\bibitem{R} Rudin, W.: {\em Real and Complex Analysis}, Mac Graw
%Hill, Second edition, 1980.


%\bibitem{P} Peller, V.~V.: {\em Hankel operators and their applications}. Springer Monographs in Mathematics.
% Springer-Verlag, New York, 2003.
% 
% \bibitem{S} Semmes, S., {\em Trace ideal criteria for Hankel operators and applications to Besov spaces,} Integral Equations and Operator Theory 7 (1984), 241--281.

%\bibitem{T1} Treil, S.~R.  : Moduli of Hankel operators and a problem of  Peller- Khrushch\"ev. (Russian)  Dokl. Akad. Nauk SSSR  283  (1985),  no. 5, 1095Ð1099.
%English transl. in Soviet Math. Dokl. 32 (1985), 293-297.

%\bibitem{T2} Treil, S.~R. : Moduli of Hankel operators and the V. V. Peller-S. Kh. Khrushch\"ev problem. (Russian) Investigations on linear operators and the theory of functions, XIV.  Zap. Nauchn. Sem. Leningrad. Otdel. Mat. Inst. Steklov. (LOMI)   141  (1985), 39Ð55.

\bibitem{ZS} Zakharov, V.~E., Shabat, A.~B., {\em Exact theory of two-dimensional self-focusing and
one-dimensional self-modulation of waves in nonlinear media.} Soviet
Physics JETP  34  (1972), no. 1, 62--69.

\end{thebibliography}
\end{document}